\documentclass{amsart}

\numberwithin{equation}{section}
\newcommand{\diam}{\text{diam\ }}
\newcommand{\supp}{\text{supp\ }}

\usepackage{indentfirst}
\usepackage{amsmath}
\usepackage{amscd}
\usepackage{amssymb, cancel, enumerate,color}
\usepackage{amsthm}
\usepackage[latin1]{inputenc}  
\newtheorem{theorem}{Theorem}[section]
\newtheorem{lemma}[theorem]{Lemma}

\theoremstyle{definition}

\newtheorem{proposition}[theorem]{Proposition}

\theoremstyle{remark}
\newtheorem{remark}[theorem]{Remark}

\numberwithin{equation}{section}

\begin{document}

\title[Riesz Type Potentials with upper doubling measures]
{Riesz Type Potentials in the framework of quasi-metric spaces equipped with upper doubling measures}

\author[B. Iaffei]{Bibiana Iaffei}
\address{Departamento de Matem\'atica (FHUC-UNL), IMAL-CONICET, Santa Fe, Argentina}

\email{biaffei@santafe-conicet.gov.ar}%

\author[L. Nitti]{Liliana Nitti}
\address{Departamento de Matem\'atica (FHUC-UNL), IMAL-CONICET, Santa Fe, Argentina}

\email{rnitti@fhuc.unl.edu.ar}


\thanks{The authors were supported in part by CONICET, CAI+D(UNL) and ANCPyT}



\begin{abstract}
The purpose of this paper is threefold. First  the natural extension of Riesz potentials to the context of quasi metric measure spaces for the class of upper doubling measures are studied on Lebesgue spaces, obtaining necessary and sufficient conditions on a  upper doubling  measure. Second, we exhibit a geometric property of the measure of the ball which permit prove the boundedness in a unified way, both in of the doubling as  non doubling situation. Third,  we  show that  the result can be applied to a type Riesz potential operator defined over a space formed by two components which are not of necessarily equal dimensions.
\end{abstract}

\maketitle

\maketitle

\section{Introduction} 
\label{intro}

We study generalized potential operators with kernels $\frac{ d(x,y)^{\alpha}}{\lambda(x,d(x,y))}$ on  a bounded quasi-metric space $(X,d)$ with an upper doubling measure $\mu$. These Riesz type potential operators denoted $I_{\alpha}^{\lambda}$ are defined by the formula
\begin{equation}\label{defIalfamulambda0}
I_{\alpha}^{\lambda} f(x)
= \int_{X} \frac{ d(x,y)^{\alpha}}{\lambda(x,d(x,y))} f(y)\  d\mu(y),
\end{equation}
for  $\alpha>0$ and a function $\lambda:X \times \mathbb{R}^+\to \mathbb{R}^+$ that is doubling and as function of the second variable is non-decreasing. Under the assumption that the function $\lambda$, as function of the second variable, is of lower type greater than $\alpha$, definition \eqref{defIalfamulambda0} makes perfectly good sense when $f\in L^p(X,d,\mu)$.

We consider so-called upper doubling measures $\mu$, introduced in \cite{Hy10}, which constitute a simultaneous generalization of doubling measures and those with the upper power bound property $\mu(B(x,r)) \leq Cr^n$, which are the ones usually called in the literature as non-doubling measures; even though note that power bounded measures are only different, not more general than, the doubling measures.

The present work is devoted to investigate the behavior on Lebesgue spaces of the  Riesz potential type operator $I_{\alpha}^{\lambda}$ associated to an upper doubling measure $\mu$  with a dominating function $\lambda$ on a  bounded quasi-metric space.

Notice  that in the classical case when $X=\mathbb{R}^n$, $d$ is the usual Euclidean metric on $\mathbb{R}^n$,  $\mu$ the Lebesgue measure on $\mathbb{R}^n$ and $\lambda(x,r)=r^n$, the basic operator of potential type is the usual fractional integral  operator $I_{\alpha}$  given by the formula $$I_\alpha f(x)= \int_{\mathbb{R}^n} \frac{f(y)}{|x- y|^{n-\alpha}}\ dy.$$
For $\alpha = 2$, it is known as the Newtonian potential.
The fractional integral operator was introduced by Hardy and Littlewood \cite{HL-I,HL-II} and Sobolev \cite{Sob}. They proved that
$I_{\alpha}$ apply a function in $L^p$  boundedly to a function in $L^q$, provided that $\frac{1}{q}=\frac{1}{p}-\frac{\alpha}{n}$. Also this result was proved using interpolation in \cite{St} and a pointwise estimation involving maximal function in \cite{Hed}.

Norm estimates for $I_{\alpha} f$ on Lebesgue spaces, as well as for operators with more general kernels than $|x-y|^{\alpha-n}$ and defined on more general spaces, have been extensively studied.

Early investigations in the direction of obtaining a high generality was inspired by the  fact that  the doubling property of the Lebesgue measure play an important role. Indeed, Gatto and Vági in \cite{GaVa1} proved the Hardy-Littlewood-Sobolev theorem in the context of space of homogeneous type if is imposed to the space the condition of normality. The theorem is an easy consequence of an inequality of Hedberg (see \cite{Hed}) which was shown  that hold in normal homogeneous type spaces. Nakai in \cite{Na2} obtained the result when the space is called $Q$-homogeneous or Ahlfors $Q$-regular metric measure space.

The boundedness of $I_{\alpha}$ on Lebesgue spaces in the non-homogeneous setting has been studied by García Cuerva and Gatto in \cite{GaCuGa}, García Cuerva and Martell in \cite{GaCuMa01} and Kokilashvili and Meshki in \cite{KoMe1}, \cite{EKM}. In the first work was used interpolation and  in the others Hedberg's inequality. If the $L^p(\mu)-L^q(\mu)$  boundedness of $I_{\alpha} f$ depends on the boundedness of the Hardy-Littlewood maximal operator on $L^p(\mu)$  and $\mu$ is a measure non-doubling, it must be modified its definition somehow, because as shown in \cite{NTV98} the usual Hardy-Littlewood maximal operator is not bounded in $L^p(\mu)$.

Fractals are measurable sets with non-integer Hausdorff dimensions.  The interest in problems of the operator theory in fractal sets has been growing continuously during the last few years because of numerous applications in another sciences.
In the fractal context, Riesz potentials are considered in \cite{Za04}, \cite{Za05}  where are introduced as traces of the corresponding Euclidean variants. A related approach for   $(X,d,\mu)$ by means of local Euclidean charts can be found in \cite{TrYa01}. In particular, the Riesz potentials of order $\alpha$ on so-called $s$-sets  are given by
$$I_\alpha^s f(x)= \int_{\mathbb{R}^n} \frac{f(y)}{d(x,y)^{s-\alpha}} d\mu(y),$$
where the rol of dimension is played by $s$.

Our work will extend  the Riesz potential in the more general situation where there may be many (non-overlapping) fractals embedded in $\mathbb R^n$, with different or not Hausdorff dimensions and which may touch each other or not. In this context where the variable dimension is permitted, we must point out the work of  Hambly and Kumagai \cite{HaKu}, who studied diffusion processes on fractals components embedded in  $\mathbb{R}^2$. One can find physical examples relative to this subject in \cite{AFPf}, \cite{FVA} and \cite{Ta}.

The goal of the present paper is threefold. First we extend the Riesz potentials  in one different direction that including both the doubling and non-doubling situation in a unified way, all the more we give necessary and sufficient conditions on a  measure for which the estimate for $I_{\alpha}^{\lambda}$  on Lebesgue spaces holds, generalizing the result in \cite{GaCuGa} and \cite{KoMe1}. Second, we provide the proof of the boundedness adapting an  idea from \cite{KoMe1}  that  allow us to unify the proof both for the homogeneous and non-homogeneous case. The basic strategy is consider an adequate  maximal function which involves the measure of the balls and the dominating function $\lambda$. Third, as an application of our result and after recognizing  that the measure $\mu^{\gamma_1,\gamma_2}$ defined in \cite{AN} satisfies $\mu^{\gamma_1,\gamma_2}(B(x,r))\leq C r^{n(x)}$ for all $x\in X$ and $r>0$ , we derive norms estimates for appropriate generalized fractional integrals in the context of a metric space,  formed by two components  with a contact of order zero, and such that each component supports a Ahlfors $n_i$-regular measure, $i=1,2$ ($n_1$ not necessarily equal to $n_2$). More precisely, we consider the operator
\begin{equation}\label{defIalfamug12}
I_{\alpha}^{n(\cdot)} f(x)= \int_{X} \frac{ d(x,y)^{\alpha}}{d(x,y)^{n(x)}} \, f(y) \, d\mu^{\gamma_{1},\gamma_{2}}(y),
\end{equation}
where $n(x)$ can take the values $n_1$ or $n_2$ and $\mu^{\gamma_{1},\gamma_{2}}$ is an upper doubling measure with $\lambda(x,r)=r^{n(x)}$. We find necessary and sufficient conditions for the inequality $\|I_{\alpha}^{n(\cdot)} f\|_{q(\cdot)}\leq C \|f\|_p$ be true. Observe that the operator $I_{\alpha}^{n(\cdot)}$ applies $L^p$ in a
variable Lebesgue space $L^q(\cdot)$.

In \cite{HHL} was considered the variable dimension and  was proved the boundedness on variable Lebesgue spaces of the following operator
\begin{equation}\label{defIalfavar}
I_{\alpha} f(x)= \int_{X} \frac{ d(x,y)^{\alpha}}{\mu(B(x,d(x,y)))} \, f(y) \, d\mu (y),
\end{equation}
by assuming that the measure is lower  Ahlfors $Q(\cdot)$-regular in a bounded subset $X$ of $\mathbb{R}^n$. It is immediately clear that the doubling condition for a measure is stronger that lower Ahlfors regularity, but weaker than Ahlfors regularity.

We note that \eqref{defIalfamug12}, unlike of \eqref{defIalfavar},  carries information about the behavior of the dimension.

Let us describe our setting  in more details and  submit the definitions of the basic concepts in Section \ref{setting}.
Section \ref{operator} is devoted to introduce the appropriate  Riesz potential operator $I_{\alpha}^{\lambda}$ in an upper doubling environment.
In Section \ref{boundedness} we give  necessary and sufficient conditions on  the measure for which the boundedness for $I_{\alpha}^{\lambda}$  on Lebesgue spaces holds. In Section \ref{doscomponentes} we study regularity properties of one measure defined in \cite{AN} and show that this measure is  another non-trivial example  of upper doubling measure. In section \ref{Rieszdoscomponentes} we state necessary and sufficient conditions on an  upper doubling  measure for which the boundedness for $I_{\alpha}^{n(\cdot)}$  on Lebesgue spaces holds.

\section{The general setting and basic facts}\label{setting}
\subsection{ Quasi-metric measure spaces}
Let $(X,d)$ a quasi-metric space. By a quasi-metric on a set $X$ we mean a nonnegative function $d$ defined on $X\times X$ such that
\begin{equation*}
d(x,y) \geq 0 \ \text{for\ every\ $x$\ and\ $y$\ in\ $X$} \: \text{and \ } \: d(x,y)=0\ \text{if\ and\ only\ if}\ x=y,
\end{equation*}
\begin{equation*}
d(x,y)=d(y,x), \text {for\ every\ $x$\ and\ $y$\ in\ $X$}
\end{equation*}
\begin{equation*}
d(x,y)\leq K_1 (d(x,z)+ d(z,y)),
\end{equation*}
for every $x$ and $y$ in $X$ and for some finite constant $K_1>0$.

A quasi-metric space $(X, d)$ is \textit{geometrically doubling} or has the \textit{weak homogeneity property} if there exists a natural number $N$ such that every open $d$-ball $B(x, r) = \{y \in X : d(y, x) < r\}$ can be covered by at most $N$ balls of radius $r/2$. A basic observation is that in a geometrically doubling quasi-metric space, a ball $B(x, r)$ can contain the centers
$x_i$ of at most $N\alpha^{-n}$ disjoint balls $B(x_i, \alpha r)$ for $\alpha\in (0, 1]$.
This weak homogeneity was first observed by Coifman and Weiss in \cite{CW}. Hyt\"onen in \cite{Hy10} gives another equivalent conditions of that definition. As was shown by Mac\'ias and Segovia in \cite{MS79}, every quasi-metric space is metrizable in
the sense that there exist a distance $\rho$ and a positive number $\alpha$ such that $\rho^{\alpha}$ is
equivalent to $d$.
The geometrically doubling quasi-metric spaces also satisfy the following topological properties: are separable and have the Heine-Borel property (see \cite{A1}).

A Borel measure $\mu$  defined on the $d$-balls  is said to be non-trivial if  $\mu(B(x,r))$ is positive and finite for every $x\in X$ and every $r > 0$.  A non-trivial measure $\mu$ is said to be is \textit{doubling}, if there exists a positive constant $K_2$ such that the
inequalities
\begin{equation}\label{duplicacion}
 \mu(B(x,2r))\leq K_2\mu(B(x,r)),
\end{equation} hold for every
$x\in X$ and every $r>0$. We say that $(X,d,\mu)$ is a \textit{space of homogeneous type} if $\mu$ is doubling on $(X,d)$. There is an
extensive literature on analysis on these structures, and several examples and
applications are given in \cite{CW}.

It is well known that if $(X,d)$ supports a doubling measure  then $(X,d)$ is geometrically doubling. Indeed, it was one of the first things pointed out by Coifman and Weiss in \cite{CW} (p. 67). Luukkainen and Saksman \cite{LuSa} proved that if $(X,d)$ is a complete, geometrically doubling metric space, then there exists a Borel measure $\mu$ on $X$ such that $(X,d,\mu)$ is a space of homogeneous type. Also a compact metric space carries a non-trivial doubling measure if and only if it is geometrically doubling metric space \cite{VK2}, \cite{Wu}.

We say that a point $x$ in a space of homogeneous type $(X,d,\mu)$ is an atom if $\mu({x}) > 0$. When $\mu({x}) = 0$ for every $x\in X$ we say that $(X,d,\mu)$ is a non-atomic space. Mac\'ias and Segovia in \cite{MS79}  proved, in the context of space of homogeneous type, that a point is an atom if and only if it is topologically isolated, and that the set of such points is at most countable.

In this article we  assume that $\text{diam}\ (X) < \infty$, then there exists a nonnegative constant $R_0=\text{diam\ } X$ such that
\begin{equation}\label{acotacionespacio}
X=B(x,R_0)
\end{equation}
for all $x\in X$.

As is known, from \eqref{duplicacion}  follows the
property
\begin{equation}\label{cocmedidas}
\frac{\mu(B(x, \rho))}{\mu(B(y, r))}\geq C_{\mu} \biggl(\frac{\rho}{r}\biggr)^{N}\quad  N = \log _2 K_2;
\end{equation}
for all the balls $B(x, \rho)$ and $B(y, r)$ with $0 < r \leq \rho<\infty $ and $ y\in B(x,\rho),$ where $C_{\mu} > 0$
does not depend on $r$, $\rho$  and $x$. From \eqref{cocmedidas} we have
\begin{equation}\label{lower}
\mu(B(x, r)) \geq c_0 r^N; \quad x\in X;\ 0 < r \leq \text{diam}(X);
\end{equation}
 Condition
\eqref{lower} is also known as the \textit{lower Ahlfors regularity condition}. The \textit{upper Ahlfors regularity condition} (also called the \textit{non-doubling condition}) holds on $X$ if there exists $n > 0$ such that
\begin{equation}\label{nondoub}
B(x, r)\leq c_1 r^n;
\end{equation}
where $c_1 > 0$ does not depend on $x\in X$  and $0 < r \leq \text{diam}(X)$, and $n$ need not to
be an integer.

Given a  Borel measure $\mu $ on $X$, we say that $(X,d, \mu)$ is  an \textit{Ahlfors $Q$-regular
metric measure space} or $Q$-normal space, for $Q>0$, if there exists a constant $A_1\geq 1$,
such that,
\begin{equation}\label{normalidad}
A_1^{-1}r^Q
\leq \mu(B(x, r)) \leq A_1 r^Q,
\end{equation}
for $0 < r \leq \text{diam}(X)$  and $x\in X$.
It is easy to show that if $(X,d,\mu)$ is an Ahlfors $Q$-regular quasi-metric measure space, then the
Hausdorff dimension, with respect to $d$, is exactly $Q$. Moreover for $Q>0$ no upper Ahlfors $Q$-regular
quasi-metric measure space has atoms in the sense that no single point has positive $\mu$ measure. In particular, if $\mu$ is positive on the balls and  satisfies the upper Ahlfors $Q$-regular condition, the space  have not  isolated points. Otherwise no lower Ahlfors $Q$-regular quasi-metric measure space has isolated points.

If in the above definitions is  modified  the variation interval of $r$ in the following way: $\frac{\mu(\{x\})}{A_1} < r \leq \text{diam}(X)$, is contemplated the case of bounded spaces with atoms. When $Q=1$  the space $(X,d,\mu)$ is named normal space (see \cite{MS79}), if $Q\neq 1$ the space is usually called $Q$-normal.

It is of interest to study also spaces with a variable dimension. Thus, if $Q: X \to (0,\infty)$ is a bounded function,
then we say that $\mu$ is \textit{Ahlfors $Q(\cdot)$-regular} if $\mu(B(x, r))\approx r^{Q(x)}$,  for all $x\in X$ and $0<r\leq \diam X$. Ahlfors $Q(\cdot)$-regularity is only possible for sufficiently regular functions $Q$ (see \cite{HHL}).  It can be defined similarly that the measure $\mu$  is \textit{lower Ahlfors $N(\cdot)$-regular} if $\mu(B(x; r)) \geq c_0' r^{N(x)}$ or \textit{upper Ahlfors $n(\cdot)$-regular} if $\mu(B(x; r)) \leq c_1' r^{n(x)}$  for all $x\in X$ and $r\in (0, 1)$.

Hyt\"onen in \cite{Hy10} defines a class of measures that encompasses both the doubling measures and those satisfying the upper power bound  $\mu(B(x; r)) \leq c_1 r^{n}$ or  $\mu(B(x; r)) \leq c_1' r^{n(x)}$.  Namely measures which are controlled from above by functions doubling.
More precisely, a Borel measure $\mu$ in some quasi-metric space $(X, d)$ is called \textit{upper doubling} if there exists a dominating function
$\lambda: X \times \mathbb{R}^{+} \to \mathbb{R}^{+} $ so that $r \to \lambda(x, r)$ is non-decreasing, $\lambda(x, 2r) \leq C_{\lambda}\lambda(x, r)$ and
\begin{equation}\label{updoub}
\mu(B(x, r)) \leq \lambda(x, r) \quad \text{for all\ } x \in X \text{\ and\ } r > 0.
\end{equation}
A quasi-metric measure space $(X,d,\mu)$ is said to be upper doubling if $\mu$ is a measure upper doubling.
The number $d := \log_2 C_{\lambda}$ can be thought of as (an upper bound for) a dimension of the measure $\mu$, and it  plays a similar role as the quantity denoted by $N$ in \eqref{cocmedidas}.
It was proved in \cite{HyYY} that there exists another dominating function $\tilde \lambda$ such that $\tilde \lambda\leq \lambda$, $C_{\tilde\lambda}\leq C_{\lambda}$ and,  for all $x, y \in X$ with $d(x, y)< r$,
\begin{equation}\label{laotralambda}
\tilde\lambda(x, r) = C_{\tilde\lambda} \tilde\lambda(y, r).
\end{equation}
Thus in what follows, we always assume that $\lambda$ satisfies \eqref{laotralambda}.

It is immediate that a measure  doubling is a special case of upper doubling, where one can  take the dominating function to be $\lambda(x,r)=\mu(B(x,r))$. On the other hand, a non-doubling measure is upper doubling with $\lambda(x,r)= Cr^n$.  Hyt\"onen and Martikainen in \cite{HM12} note that the measures obtained by Volberg and Wick in \cite{VW12}  are actually upper doubling. In the section \ref{doscomponentes} we show that the measure defined by one the authors and Aimar in \cite{AN} is another non-trivial example of upper doubling measure and moreover is a upper Ahlfors $n(\cdot)$-regular measure.

In the following lemma we state a relation between the upper doubling measures and the atoms.
\begin{lemma}\label{upperandatom}
If $\mu$ is a upper doubling measure on $X$ with a dominating function $\lambda$ which satisfies that $\lambda(x,r_j)\to 0$ for each $x$ when $r_j\to 0$ for $j\to \infty$, then  $\mu$ not have atoms. If additionally the measure is  positive on the balls, is obtained that the space have not   isolated points.
\end{lemma}
\begin{proof}
Suppose on the contrary that there exists $x\in X$  such that $\mu(\{x\})=\alpha>0$, but $\alpha<\mu(\{x\})\leq \mu(B(x,r_j)\leq \lambda(x,r_j)$ and then using the hypotheses we obtain $\alpha<0$, contradicting the assumption about $\alpha$. The proof of second part is immediately from the fact that  $x$ is isolated point there is a number positive $R$ such that $\{x\}=B(x,R)$.

\end{proof}
Note that the above property on $\lambda$ is satisfied in the case $\lambda(x,r)=Cr^n$ and when $\lambda(x,r)=\mu(B(x,r))$ is translated as $\mu(B(x,r_j))\to 0$ for each $x$ when $r_j\to 0$ for $j\to \infty$. The measure $\mu^{\gamma_1,\gamma_2}$ introduced in section \ref{doscomponentes} is another upper doubling measure with dominating function that satisfies the previous condition.

We refer to \cite{A1}, \cite{EKM}, \cite{GGKK}, \cite{HaKo}, \cite{He} for general properties of quasi-metric measure
spaces.

\subsection{The modified maximal operator}
Let $(X,d)$ be a   geometrically doubling  quasi-metric space and $\mu$ be a Borel measure on $X$ which is finite on bounded sets.
Recall that the Hardy-Littlewood maximal function $Mf(x)$ is defined (for Borel measurable functions $f$) by
$$Mf(x) := \sup_{r>0} \frac{1}{\mu(B(x,r))} \int_{B(x,r)} |f|\ d\mu.$$
The definition makes sense $\mu$-almost everywhere since
 if $x\in \supp \mu$, then $\mu (B(x, r))$ is positive for every $r > 0$ (otherwise a small open
ball centered at $x$ could be removed from the support of $\mu$).
If the measure $\mu$ satisfies the doubling property,  the Hardy-Littlewood maximal operator is well-known
to be bounded on all $L^p(\mu)$ with $1 < p \leq +\infty$ and from $L^1(\mu)$ to $L^{1,\infty}(\mu)$. But, omitting the doubling requirement, for arbitrary  geometrically doubling quasi-metric space $X$ and measure $\mu$, only we can said
 that $M$ is bounded on $L^{\infty}(\mu)$. One way to avoid this problem is to replace the measure of the ball $B(x, r)$ in the denominator by the measure of the three times larger ball, i.e., to define
\begin{equation}\label{maximalmodificada}
\tilde{M}f(x):= \sup_{r>0}\frac{1}{\mu(B(x,3K_1 r))} \int_{B(x,r)} |f|\ d\mu,
\end{equation}
where the constant $K_1$ is from of definition of a quasi-metric.
Note that always $\tilde{M}f(x)\leq Mf(x)$ and, if the measure $\mu$ satisfies the doubling condition,
$\tilde{M}f(x)\leq C Mf(x)$ for some constant $C > 0$.
\begin{lemma}\label{acotacionmaximal}
 If $(X,d)$ is geometrically doubling, and $\mu$ is a Borel measure on $X$ which is finite on bounded sets, the modified maximal  operator $\tilde{M}$ is bounded on $L^p(\mu)$ for each $p\in (1;\infty]$ and acts from $L^1(\mu)$ to $L^{1,\infty}(\mu)$.
\end{lemma}
The weak type 1-1 estimate have been proved by Nazarov, Treil and Volverg in \cite{NTV98}.
For other approach see \cite{EKM}, p. 368, and the references therein.

\subsection{Variable exponent Lebesgue spaces}
We refer  in the next subsection the basic definitions  and properties of variable exponent Lebesgue spaces which appear in Sections \ref{boundedness} and \ref{Rieszdoscomponentes}.

Let $p : X \to [1;\infty)$ be a $\mu$-measurable function. Everywhere below we assume
that
\begin{equation}\label{pmenospmas}
1 < p_{-}\leq p(x)\leq p_{+} < \infty; \  x \in X;
\end{equation}
according to the notation in (2.1). By $L^{p(\cdot)}(X)$ we denote the space of all $\mu$
measurable functions $f$ on $X$ such that the modular
\begin{equation}\label{modular}
I_{p(\cdot)}(f) = I_{p(\cdot);X} (f) :=\int_{X} |f(x)|^{p(x)} d\mu(x)
\end{equation}
is finite. This is a Banach space with respect to the norm
\begin{equation}\label{normapvariable}
\|f\|_{p(\cdot)} = \|f\|_{p(\cdot);X} := \inf\{\lambda > 0 : I_{p(\cdot)}\biggl(\frac{f}{\lambda}\biggr)\leq 1\}
\end{equation}

It can be seen  in \cite{KR} that $I_{p(\cdot)}$ has the following properties:
\begin{enumerate}[(i)]
\item $I_{p(\cdot)}(f)\geq 0$ for every function $f$.
\item $I_{p(\cdot)}(f)=0$ if and only if $f=0$.
\item $I_{p(\cdot)}(-f)= I_{p(\cdot)}(f)$ for every $f$.
\item $I_{p(\cdot)}$ is convex.
\item If $ |f(x)| \geq |g(x)|$ for a.e. $x\in X$ and if  $I_{p(\cdot)}(f)<\infty$, then
 $I_{p(\cdot)}(f)\geq I_{p(\cdot)}(g)$; the last inequality is strict if $ |f|\neq |g|$.
\item \begin{equation}\label{norm1}
\text{If\ } \|f\|_{p(\cdot)}\leq 1, \text{\ then\ }  I_{p(\cdot)}(f)\leq \|f\|_{p(\cdot)}.
\end{equation}
\end{enumerate}
The properties (i)--(iv) characterize $I_{p(\cdot)}$ as the convex modular in the sense of \cite{Mu}.

In the setting of quasi-metric measure spaces $(X,d,\mu)$ can be proved one version of the theorem 2.8 in \cite{KR} which states that $L^{q(x)}$ is continuously embedded in $L^{p(x)}$ if and only if $p(x)\leq q(x)$ for a.e. $x\in X$, when $0<\mu(X)<\infty$ and $p,q$ are measurable functions such that $p,q:X\to [1,\infty)$. This result will be used in theorem \ref{necIlambda}.

Variable exponent Lebesgue spaces on general quasi-metric measure spaces have
been considered in \cite{FMS}, \cite{HHL}, \cite{HHP}, \cite{AS}  and \cite{GPS}.

\subsection{Lower and upper type  functions}
Recall some  definitions  concerning to increasing functions which  appear in the bibliography when is attempted  to  generalize  power functions.

We say that one such function $\lambda$ is of \textit{lower type} $a\geq 0$ if
\begin{equation}
\lambda(st)\leq c_1 s^a \lambda(t)
\end{equation}
for some constant $c_1$, every $0<s\leq 1$ and every $t>0$. Similarly
$\lambda$ is of \textit{upper type} $b\geq 0$ if
\begin{equation}
\lambda(st)\leq c_2 s^b \lambda (t)
\end{equation}
for some constant $c_2$, every $s\geq 1$ and every $t>0$.
It is immediate that if $\lambda$ is of lower type $a_1$ and $a_2<a_1$ then $\lambda$ is also of lower type $a_2$. We say that a function is of \textit{lower type greater than} $\alpha$ if it is of lower type $\alpha_0$, for some $\alpha_0>\alpha$. Similarly for \textit{upper type less than }$\alpha$.
The  types of a function determine an infinite ray, this allows us introduce the concept of upper and lower index as the infimum and supremum, respectively, of such sets.


Now we  state one property that will be useful in the following sections.
\begin{proposition}\label{proptipoinferior}
Let the function $\lambda$  and the positive real number $\alpha$. The function $\lambda$ is of lower type  $\alpha$ if and only if the inequality
\begin{equation}\label{equitipo}
\frac{r_2^{\alpha}}{\lambda(r_2)}\leq c_1 \frac{r_1^{\alpha}}{\lambda(r_1)}
\end{equation}
holds for all $0< r_1\leq r_2$ and some positive constant $c_1$.
\end{proposition}
\begin{proof}
Suppose first \eqref{equitipo} holds and $0< s\leq 1$, so that  $st\leq t$ for all $t>0$. Then \eqref{equitipo} takes the following form:
\begin{equation*}
\frac{t^{\alpha}}{\lambda(t)}\leq c_1 \frac{(st)^{\alpha}}{\lambda(st)}.
\end{equation*}
From this we obtain that $\lambda$ is of lower type  $\alpha$. On the other hand if $\lambda$ is of lower type  $\alpha$ and we assume $0< r_1\leq r_2$ then
\begin{equation*}
\lambda (r_1)=\lambda \Bigl({\frac{r_1}{r_2}} \:  r_2\Bigr)\leq c_1 \Bigl(\frac{r_1}{r_2} \Bigr)^{\alpha}\lambda(r_2),
\end{equation*}
and \eqref{equitipo} is satisfied.
\end{proof}

\section{The Riesz operator in a upper doubling environment }\label{operator}
The fractional integral operator or the Riesz potential $I_{\alpha}$, $0 < \alpha < n$, is defined by
\begin{equation}\label{defIalfa}
I_{\alpha} f(x)= \int_{\mathbb{R}^n}  \frac{f(y)}{|x-y|^{n-\alpha}} dy,
\end{equation}
$x\in \mathbb{R}^n$,
for any suitable function $f$ on $\mathbb{R}^n$. Clearly $I_{\alpha} f$ is well-defined for any locally bounded
function $f$ on $\mathbb{R}^n$. This operator was first studied by Hardy and Littlewood in the 1920's \cite{HL-I, HL-II}
 and extended by Sobolev \cite{Sob} in the 1930's. A well-known result for $I_{\alpha} f$ is the Hardy-Littlewood-Sobolev inequality: $\|I_{\alpha}f\|_{L^q} \leq C_{p,q} \|f\|_{L^p}$. That is, $I_{\alpha} f$ is bounded from $L^p(\mathbb{R}^n)$ to $L^q(\mathbb{R}^n)$ if and
only if $\frac{1}{q}=\frac{1}{p}-\frac{\alpha}{n}$, with $1 < p < \frac{n}{\alpha}$. (see e.g. \cite{St}).

These statements were generalized in many directions, for historical notices
and review of results see the book \cite{Sabook}.

As mentioned in the introduction, also fractional integrals over quasi-metric measure spaces are known to be considered in different forms.  There are natural  extensions to contexts of  quasi-metric measure spaces that arise from considering $|x-y|^n=\mu(B(x,|x-y|)) $, where $\mu$ is the $n$-dimensional Lebesgue measure or simply  $|x-y|^n$ as $n$-dimensional power of the  Euclidean distance between $x$ e $y$, or dealing $|x-y|^n$ as a quasi-distance between $x$ and $y$.

In what follows, we shall assume that $(X,d)$ is a  geometrically doubling  quasi-metric space, the $d$-balls are open sets,
 $\mu$ be a Borel measure on $X$ which is finite on bounded sets, positive on the balls and  $\mu(\{x\}) = 0$ for all $x \in X$.

We consider the following operators of potential type:
\begin{equation}
I_{\alpha}^Q f(x)= \int_{X} f(y) \frac{1}{d(x,y)^{Q-\alpha}} d\mu(y), \quad 0<\alpha<Q\leq n.
\end{equation}

\begin{equation}
I_{\gamma} f(x)= \int_{X} f(y) \frac{1}{d(x,y)^{1-\gamma}} d\mu(y), \quad 0<\gamma<1.
\end{equation}

\begin{equation}
K_{\gamma} f(x)= \int_{X} f(y) \frac{1}{\mu (B(x, d(x,y)))^{1-\gamma}} d\mu(y), \quad 0<\gamma<1.
\end{equation}

\begin{equation}
J_{\alpha} f(x)= \int_{X} f(y) \frac{d(x,y)^{\alpha}}{\mu (B(x, d(x,y)))} d\mu(y), \quad \alpha>0.
\end{equation}

We observe that from the results obtained in \cite{MS79}, it turns out  to be that given an arbitrary space of homogeneous type $(X, d,\mu)$, there exists a normal space $(X,\delta,\mu)$ of orden $\theta$, $\theta>0$,  such that the $L^p(X,d,\mu)$ coincides with $L^p(X,\delta,\mu)$. Then in the case $\mu$ is doubling, the study  on Lebesque spaces the boundedness of $I_{\gamma} f(x)$  or  $K_{\gamma} f(x)$ is indistinct, because  both operators are equivalents.

Obviously, if $\mu$ is  Ahlfors $Q$-regular, then $I_{\alpha}^Q f(x)$ and $J_{\alpha} f(x)$ are equivalents. This is what happens for example if we  consider the case of $s$-sets and $\mu$ is the restriction of the Hausdorff $s$-measure $\mathcal{H}^s$ to these sets.

If $\mu$ is doubling using \eqref{lower} we have $J_{\alpha} f(x)\leq \frac{1}{c_0}I_{\alpha}^N f(x)$, $f\geq 0$. Similarly,
$I_{\alpha}^n f(x)\leq J_{\alpha} f(x)$, $f\geq 0$, when $\mu$ is ``non-doubling'', i.e. \eqref{nondoub} holds. Moreover can be seen (\cite{KoMe1}, \cite{GaCuGa}) that for a measure $\mu$, finite over balls and not having any atoms, condition \eqref{nondoub} is necessary for the inequality $\|I_{\alpha}^n\|_q\leq C \|f\|_p, \frac{1}{q}=\frac{1}{p}-\frac{\alpha}{n}$ to hold.

In the general case, $c_0 r^N\leq \mu(B(x,r))\leq c_1 r^n$, where $n\leq N$ and $r\in (0,1)$, the operator $J_{\alpha}$ is better suited for lower Ahlfors $N$-regular quasi-metric measure spaces, and $I_{\alpha}^n$  is better adjusted for upper Ahlfors $n$-regular quasi-metric measure spaces.

The four potential type integral operators defined above can be viewed as special cases of the following operator
\begin{equation}\label{defIalfamulambda}
I_{\alpha}^{\lambda} f(x)= \int_{X} \frac{ d(x,y)^{\alpha}}{\lambda(x,d(x,y))} \, f(y) \, d\mu(y),
\end{equation}
where $\lambda$ is a dominating function for the  upper doubling measure $\mu$ and as function of the variable $r$ is of lower type greater than $\alpha$, whenever this integral is finite. Clearly, if $f$ is a bounded function with compact support, then the integral in \eqref{defIalfamulambda} is finite for almost every $x\in X$, according to \eqref{acotacionespacio}  we have $X=B(x,R_0)$.

\section{Boundedness of $I_{\alpha}^{\lambda}$ in Lebesgue spaces}\label{boundedness}
Now we  state one of our main results, which gives  a version of Hardy-Littlewood-Sobolev inequality in the context of upper doubling space. We describe those measure spaces with quasi-metrics on which the potential type operator maps $I_{\alpha}^{\lambda}:L^p(X,d,\mu) \to  L^{q(\cdot)}(X,d,\mu)$  boundedly.

\begin{theorem}\label{sufIlambda}
 Let $\alpha >0$, $1<p<q_-\leq q(x)\leq q_+<\infty$  for all  $x\in X$. We assume that  $(X,d)$ is  a  bounded  geometrically doubling quasi-metric space, such that the $d$-balls are open sets, and there exists a function $\lambda:X \times \mathbb{R}^+\to \mathbb{R}^+$ that is doubling and as function of the variable $r$ is non-decreasing and of lower type greater than $\alpha$.    If   $r^{\alpha}\leq \lambda(x,r)^{\frac{1}{p}-\frac{1}{q(x)}}$ and $\mu$ is a  Borel measure on $X$ which is finite on bounded sets, positive on the balls, $\mu(\{x\}) = 0$ for all $x \in X$ and upper doubling with dominating function $\lambda$, then $I_{\alpha}^{\lambda}$ is a bounded operator from $L^p(X,d,\mu)$ to
$L^{q(\cdot)}(X,d,\mu)$.
\end{theorem}
\begin{proof}
We are going to adapt to our context the proof given by Hedberg in \cite{Hed}.
Let $B=B(x,r)$, $x\in X$ and $r>0$. For $f\in L^p$ we write $f$ as  $f=f\chi_B + f \chi_{B^c}$, where $\chi_B$ is the characteristic function on the ball and $\chi_{B^c}$ the characteristic function on the complement of the ball.  Then we have,
\begin{align*}
I_{\alpha}^{\lambda} f(x)
&= \int_{X} \frac{ d(x,y)^{\alpha}}{\lambda(x,d(x,y))}\ f(y)\  d\mu(y)\\
&= \int_{B} \frac{ d(x,y)^{\alpha}}{\lambda(x,d(x,y))}\ f(y)\  d\mu(y) +
\int_{B^c} \frac{ d(x,y)^{\alpha}}{\lambda(x,d(x,y))}\ f(y)\  d\mu(y) \\
&=I_1+ I_2.
\end{align*}
We estimate the first integral considering one decomposition of the ball $B$ in concentric annuli that we denote $C_j=B(x, 2^{-j+1}r)- B(x, 2^{-j}r)$, $j=1,2,\ldots$ . We obtain
\begin{equation*}\label{I1}
|I_1|\leq \sum_{j=1}^{\infty}\int_{C_j}  |f(y)|  \frac{d(x,y)^{\alpha}}{\lambda(x,d(x,y))}\ d\mu(y),
\end{equation*}
We note that by virtue the hypothesis about the lower type of the function $\lambda$ we can assure that $\lambda$ is of lower type $\alpha$,  applying the property given in the Proposition \ref{proptipoinferior}, we have
\begin{equation*}
\frac{r^\alpha}{\lambda(x,r)}\leq c_1\frac{s^\alpha}{\lambda(x,s)}, \quad s<r \text{\ and\ } \forall x\in X.
\end{equation*}
From this we obtain
\begin{align*}
|I_1|&\leq \sum_{j=1}^{\infty}\int_{B(x, 2^{-j+1}r)} |f(y)| c_1 \frac{(2^{-j}r)^{\alpha}}{\lambda(x,2^{-j}r)} d\mu(y)\\
&= c_1 r^{\alpha } \sum_{j=1}^{\infty} (2^{-j})^{\alpha} \frac{1}{\lambda(x,2^{-j}r)} \frac{\mu(B(x, 3 K_1 (2^{-j}r)))}{\mu(B(x, 3 K_1 (2^{-j}r)))} \int_{B(x, 2^{-j+1}r)} |f(y)| d\mu(y)\\
&\leq c_1 r^{\alpha } \sum_{j=1}^{\infty} (2^{-j})^{\alpha} \frac{\lambda(x, 3 K_1 (2^{-j}r))}{\lambda(x,2^{-j}r)} \frac{\mu(B(x, 3 K_1 (2^{-j}r)))}{\lambda(x, 3 K_1 (2^{-j}r))} \tilde M f(x)\\
&\leq c_1r^{\alpha } \sum_{j=1}^{\infty} (2^{-j})^{\alpha} (C_{\lambda})^{\ell}  \Omega(x) \tilde M f(x)\\
& \leq  C_1 r^{\alpha } \Omega(x) \tilde M f(x),
\end{align*}
where $C_{\lambda}$ denote the constant in \eqref{updoub} and $\Omega$ is the maximal function defined by $\Omega(x)=\sup_{R>0} \frac{\mu(B(x, R))}{\lambda(x, R)}$. The constant $C_1$ depends only on $K_1$, $C_{\lambda}$, $c_1$ and $\alpha$.\\
Furthermore, as we have assumed that the underlying metric space $X$ is bounded, then there exists a constant $R_0>0$ such that $R_0={\rm diam\ } X$.
For this $R_0$  there exists $m\in \mathbb{N}_0$ such that $2^{m}r < R_0\leq2^{m+1}r$. Now set  $D_k=B(x, 2^{k+1}r)- B(x, 2^{k}r)$ and  then  we decompose $B^{c}(x,r)=\bigcup_{j=0}^{m-1} D_k \bigcup (B(x,R_0)\setminus B(x,2^{m}r))$. We estimate the second integral in a similar way considering this decomposition of the  complement of the ball $B$.
\begin{align*}
|I_2|&\leq \sum_{k=1}^{m-1}\int_{D_k} \!\!\!|f(y)|  \frac{d(x,y)^{\alpha}}{\lambda(x,d(x,y))} d\mu(y)+ \int_{B(x,R_0)\setminus B(x,2^{m}r)}\!\!\! |f(y)|  \frac{d(x,y)^{\alpha}}{\lambda(x,d(x,y))} d\mu(y)\\
&\leq \sum_{k=1}^{m-1}\int_{B(x, 2^{k+1}r)} |f(y)|  c_1 \frac{(2^k r)^{\alpha}}{\lambda(x,2^k r)} d\mu(y) + \int_{B(x,R_0)}\!\!\! |f(y)|  c_1 \frac{(2^m r)^{\alpha}}{\lambda(x,2^m r)} d\mu(y).
\end{align*}
We then apply  H\"older's inequality, in each summands, to obtain
\begin{align*}
|I_2|&\leq c_1 r^{\alpha }\Biggl[ \sum_{k=1}^{m-1} \Biggl(\int_{B(x, 2^{k+1}r)} |f(y)|^p d\mu(y)\Biggr)^{1/p} \Biggl(\int_{B(x, 2^{k+1}r)} \frac{(2^k r)^{\alpha p'}}{(\lambda(x,2^kr ))^{p'}}  d\mu(y)\Biggr)^{1/p'}\\
& \quad + \Biggl(\int_{B(x, R_0)} |f(y)|^p d\mu(y)\Biggr)^{1/p} \Biggl(\int_{B(x, R_0)} \frac{(2^m r)^{\alpha p'}}{\lambda(x,2^mr)^{p'}} d\mu(y)\Biggr)^{1/p'}\Biggr].
\end{align*}
From the hypothesis about of lower type of $\lambda(x,r)$ in the second variable, the doubling property of $\lambda$ and the definition of maximal function $\Omega$ we have

\begin{align*}
|I_2|&\leq c_1 r^{\alpha } \|f\|_p  \Biggl[\sum_{k=1}^{m-1}  \frac{(2^k )^{\alpha }}{\lambda(x,2^k r)}  \bigl(\mu(B(x, 2^{k+1}r))\bigr)^{1/p'} + \frac{(2^m)^{\alpha}}{\lambda(x,2^m r)} \bigl(\mu(B(x,R_0))\bigr)^{1/p'}\Biggr]\\
&\leq c_1 r^{\alpha }\|f\|_p \left[ \sum_{k=1}^{m-1}  C_{\lambda} \frac{2^{k\alpha}}{\lambda(x,2^{k+1} r)}   \bigl(\lambda(x, 2^{k+1}r)\bigr)^{1/p'}  + C_{\lambda} \frac{2^{m\alpha}}{\lambda(x,R_0)} \bigl(\lambda(x,R_0)\bigr)^{1/p'}\right]\\
&\leq  c_1  C_{\lambda} r^{\alpha } \|f\|_p \left[\sum_{k=1}^{m-1} 2^{k\alpha}  \bigl(\lambda(x, 2^{k+1}r)\bigr)^{-1/p} +  2^{m \alpha} \bigl(\lambda(x,R_0)\bigr)^{-1/p}  \right].
\end{align*}
Using the fact that the application $r\to \lambda(x,r)$ is non-decreasing,  we observe that $\lambda(x,r)\leq \lambda(x, 2^ {k+1}r)\leq \lambda(x, R_0), k=0, 1,\ldots, m-1$ and we conclude that
\begin{align*}
|I_2|&\leq c_1  C_{\lambda} r^{\alpha } \|f\|_p  \bigl(\lambda(x,r)\bigr)^{-1/p} \Biggl[\sum_{j=1}^{m-1} 2^{k\alpha}   +  2^{m \alpha} \Biggr]\\
&\leq C_2 r^{\alpha } \|f\|_p \bigl(\lambda(x,r)\bigr)^{-1/p},
\end{align*}
where $C_2$ is a constant that depends on $c_1$, $C_{\lambda}$, $\alpha$ and $R_0$.\\
The estimates for $I_1$ and $I_2$ imply the following pointwise inequality:
\begin{equation*}
|I_{\alpha}^{\lambda} f(x)|\leq C_3 (r^{\alpha } \Omega(x) \tilde M f(x) + r^{\alpha } \|f\|_p  \bigl(\lambda(x,r)\bigr)^{-1/p},
\end{equation*}
for arbitrary $x\in X$ and $r>0$. Taking into account condition \eqref{updoub} we deduce that $\Omega(x)\leq 1$ for all $x\in X$. Hence
\begin{equation}\label{estimacionpuntual}
|I_{\alpha}^{\lambda} f(x)|\leq C_3 (r^{\alpha } \tilde M f(x) + r^{\alpha } \|f\|_p  \bigl(\lambda(x,r)\bigr)^{-1/p} ),
\end{equation}
for arbitrary $x\in X$ and $r>0$.
Choose $r>0$ such that $\lambda(x,r)\geq \frac{\|f\|^p_p }{\tilde M f(x)^p}$. Then from inequality $r^\alpha\leq \lambda(x,r)^{\frac{1}{p}-\frac{1}{q(x)}}$ given in the hypotheses, results
\begin{equation*}
r^\alpha\leq \biggl(\frac{\tilde M f(x)}{\|f\|_p}\biggr)^{\frac{p}{q(x)}-1}.
\end{equation*}
Consequently
\begin{equation}\label{estimacionmenorR0}
|I_{\alpha}^{\lambda} f(x)|\leq C_4 \biggl( \tilde M f(x) \bigr)^{\frac{p}{q(x)}} \|f\|_p^{1-\frac{p}{q(x)}}\biggr),
\end{equation}
Note that this choice of $r$ is all right as long as $\lambda(x,r)$ does not exceed $\lambda(x,R_0)$;
however, if it does, it is because $\tilde M f(x)^p \leq \|f\|_p^{p} \,  \lambda(x,r)$ for these $x$'s, by
setting $r = R_0=\text{diam}(X)$ in \eqref{estimacionpuntual}, we see that
\begin{equation}\label{estimacionR0}
|I_{\alpha}^{\lambda}f(x)|\leq C_5  \|f\|_p.
\end{equation}
We just add \eqref{estimacionmenorR0} and \eqref{estimacionR0} and obtain that for all $x$ in $X$.
\begin{equation}\label{estimacionpuntualfinal}
|I_{\alpha}^{\lambda}f(x)|\leq C_6 \|f\|_p \bigl(\tilde Mf(x)^{\frac{p}{q(x)}} \|f\|_p^{-\frac{p}{q(x)}} +1\bigr).
\end{equation}
By \eqref{estimacionpuntualfinal} it readily follows that
\begin{equation*}\label{estimacionennorma}
\|I_{\alpha}^{\lambda}f\|_{q(\cdot)} \leq c \|f\|_p
\end{equation*}
Indeed, since $\rho$ is order preserving and convex modular (see section \ref{setting}) we get
\begin{align*}
\int_{X} &\biggl(\frac{|I_{\alpha}^{\lambda}f(x)|}{C_6 2 (C_0^p+\mu(X))^{\frac{1}{q(x)}}\|f\|_p}\biggr)^{q(x)} d\mu(x)\\
&\leq \int_{X} \biggl(\frac{\tilde Mf(x)^{\frac{p}{q(x)}} \|f\|_p^{-\frac{p}{q(x)}} +1}{2 (C_0^p+\mu(X))^{\frac{1}{q(x)}}} \biggr)^{q(x)} d\mu(x) \\
&\leq \int_{X} \biggl(\frac{1}{2}\frac{\tilde Mf(x)^{\frac{p}{q(x)}} \|f\|_p^{-\frac{p}{q(x)}}}{ (C_0^p+\mu(X))^{\frac{1}{q(x)}}}+  \frac{1}{2} \frac{1}{(C_0^p+\mu(X))^{\frac{1}{q(x)}}}\biggr)^{q(x)} d\mu(x) \\
&\leq \frac{1}{2} \frac{1}{C_0^p+\mu(X)} \Biggl(\int_{X} \frac{|\tilde Mf(x)|^{p}}{\|f\|_p^{p}} d\mu(x) + \int_{X} 1 d\mu(x)\Biggr)\\
&= \frac{1}{C_0^p+\mu(X)} \Biggl( \frac{\|\tilde Mf\|_{p}^{p}}{\|f\|_p^{p}} +\mu(X) \Biggr)\\
&\leq \frac{1}{C_0^p+\mu(X)} \Biggl( C_0^p\frac{\|f\|_{p}^{p}}{\|f\|_p^{p}} +\mu(X) \Biggr)\\
&=1,
\end{align*}
where  $C_0$ is the constant in the inequality $\|\tilde Mf\|_{p}\leq C_0 \|f\|_{p}$ from Lemma \ref{acotacionmaximal}.\\
Now
\begin{align*}
\|I_{\alpha}^{\lambda} f\|_{q(\cdot)} &\leq C_6 2 (C_0^p+\mu(X))^{\frac{1}{q(x)} }\|f\|_p\\
&\leq C_6 2 \max\{(C_0^p+\mu(X))^{\frac{1}{q_+}}, (C_0^p+\mu(X))^{\frac{1}{q_-}}\}\|f\|_p\\
&=c \|f\|_p,
\end{align*}
where we have considered that if $(C_0^p+\mu(X))\geq 1$ then $(C_0^p+\mu(X))^{\frac{1}{q(x)}}\leq (C_0^p+\mu(X))^{\frac{1}{q_-}}$ since
$\frac{1}{q(x)}\leq \frac{1}{q_-}$, but if $(C_0^p+\mu(X))<1$ then $(C_0^p+\mu(X))^{\frac{1}{q(x)}}\leq (C_0^p+\mu(X))^{\frac{1}{q_+}}$ since
$\frac{1}{q(x)}\geq \frac{1}{q_+}$. Here $q_+$ and $q_-$ correspond to the definitions given in \eqref{pmenospmas}.
\end{proof}
\begin{remark}
In the proof we have used the maximal function $\Omega(x)=\sup_{R>0} \frac{\mu(B(x, R))}{\lambda(x, R)}$ which describes a geometric property of the measure of the ball, this idea  was taken of \cite{EKM}, where was used for the case of $\lambda(x,r)=Cr^n$.
\end{remark}
\begin{remark}
If $f\in L^{p(\cdot)}$ and $p$ is a log-H\"older continuous exponent  the proof of the above Theorem can  be modified slightly and the result de boundedness of  $I_{\alpha}^{\lambda}$ continue valid in the setting of variable exponent Lebesgue.
\end{remark}

The next lemma is a variant of a lemma due to Diening \cite{Di2} and give us an estimate from below of the norm of characteristic function on the ball with measure less  or equal than one.
\begin{lemma}\label{acotacionporabajo}
Let $(X,d,\mu)$ be a quasi-metric measure space, $\mu$ a finite measure over balls  and a function $p(\cdot):X\to [1,\infty)$, $1<p_-\leq p(\cdot)\leq p_+<\infty$, then for any ball $B$ such that $\mu(B)\leq 1$,
\begin{equation}
\|\chi_{B}\|_{p(\cdot),\Omega}\geq C \mu(B)^{1/p(x)}
\end{equation}
\end{lemma}
\begin{proof}
Since $\mu(B)\leq 1$ always is true for $\lambda\geq 1$ that $\int_{B} \lambda^{-p(x)}d\mu(x)\leq\mu(B)\leq 1$, and since  $p^+<\infty$, by the definition of the norm on $L^{p(\cdot)}(X)$ we get.
\begin{align*}
\|\chi_{B}\|_{p(\cdot),\Omega} &=\inf\{\lambda >0 : \int_{B} \lambda^{-p(x)}d\mu(x)\leq 1\}\\
& =\inf\{0<\lambda<1 : \int_{B} \lambda^{-p(x)}d\mu(x)\leq 1\}\\
&\geq \inf\{0<\lambda<1 : \int_{B} \lambda^{-p_+(B)}d\mu(x)\leq 1\}\\
&=\mu(B)^{1/p_+}.
\end{align*}
Moreover
\begin{align*}
\mu(B)^{1/p_+}&=\mu(B)^{1/p(x)} \mu(B)^{1/p_+ - 1/p(x)}\\
&\geq \mu(B)^{1/p(x)}\mu(B)^{p_+ -p_+/p_+p_-}\\
&= \mu(B)^{1/p(x)}.
\end{align*}
\end{proof}
\begin{theorem}\label{necIlambda}
For a measure $\mu$, finite over balls and not having any atoms, if $r^{\alpha}=\lambda(x,r)^{\frac{1}{p}-\frac{1}{q(x)}}$  with $\alpha>0$ and $1<p<q_-\leq q(x)\leq q_+<\infty$  for all $x\in X$, then the condition $\mu(B(x,r))\leq C' \lambda(x,r)$  for some constant $C'$ is necessary for $\|I_{\alpha}^{\lambda} f\|_{q(\cdot)}\leq C \|f\|_p$ to hold,  where $\lambda$ as function of the variable $r$ is of lower type greater  than $\alpha$.
\end{theorem}
\begin{proof}
Let $I_{\alpha}^{\lambda}$ be bounded from $L^p$ to $L^{q(\cdot)}$, set $f=\chi_{B(a,r)} \frac{\lambda(\cdot,r)}{\lambda(a,r)}$, where $a\in X$, $r>0$. First we calculate
\begin{equation*}
\int_X \biggl(\chi_{B(a,r)}\!(x) \, \frac{\lambda(x,r)}{\lambda(a,r)}\biggr)^{p} d\mu(x)=\int_{B(a,r)} \biggl(\frac{\lambda(x,r)}{\lambda(a,r)} \biggr)^{p} d\mu(x)\leq\mu(B(a,r)),
\end{equation*}
since $\lambda(x,r)\leq \lambda(a,r)$ because of $d(x,a)<r$ according  to \eqref{laotralambda}. Then $\|I_{\alpha}^{\lambda} f\|_{q(\cdot)}\leq C\mu(B(a,r))^{\frac{1}{p}}$. It readily follows from \eqref{norm1} that
\begin{equation*}
\int_X \Biggl|\frac{I_{\alpha}^{\lambda}\bigl(\chi_{B(a,r)}\!(x) \, \frac{\lambda(x,r)}{\lambda(a,r)}\bigr)}{C\mu(B(a,r))^{\frac{1}{p}}}\Biggr|^{q(x)} d\mu(x)\leq 1.
\end{equation*}
For each $x\in B(a,r)$, we have
\begin{align*}
\frac{I_{\alpha}^{\lambda}\bigl(\chi_{B(a,r)}\!(x) \, \frac{\lambda(x,r)}{\lambda(a,r)}\bigr)}{C\mu(B(a,r))^{\frac{1}{p}}}&
=\frac{1}{C\mu(B(a,r))^{\frac{1}{p}}} \int_X \chi_{B(a,r)}\!(x) \, \frac{\lambda(x,r)}{\lambda(a,r)} \, \frac{d(x,y)^{\alpha}}{\lambda(x,d(x,y))} \, d\mu(y)\\
&\geq \frac{1}{C\mu(B(a,r))^{\frac{1}{p}}} \int_{B(a,r)} \, \frac{\lambda(x,r)}{\lambda(a,r)} \, \frac{(2 K_1 r)^{\alpha}}{\lambda(x,2 K_1 r)} \, d\mu(y)\\
&= \frac{(2 K_1)^{\alpha}}{C (C_{\lambda})^{\ell}} \, \frac{r^{\alpha}}{\lambda(a,r)}\mu(B(a,r))^{1-\frac{1}{p}}.
\end{align*}
Therefore
\begin{equation*}
\int_{B(a,r)}\Biggl|\frac{(2 K_1)^{\alpha}}{C (C_{\lambda})^{\ell}} \, \frac{r^{\alpha}}{\lambda(a,r)}\mu(B(a,r))^{1-\frac{1}{p}}\Biggr|^{q(x)} d\mu(x)\leq \int_X \Biggl|\frac{I_{\alpha}^{\mu_{\lambda}}\bigl(\chi_{B(a,r)}\!(x) \, \frac{\lambda(x,r)}{\lambda(a,r)}\bigr)}{C\mu(B(a,r))^{\frac{1}{p}}}\Biggr|^{q(x)} d\mu(x)\leq 1.
\end{equation*}
In consequence by definition of the norm given in \eqref{normapvariable} we get
\begin{equation}\label{estimacionmenoroigual1}
\frac{(2 K_1)^{\alpha}}{C (C_{\lambda})^{\ell}} \, \frac{r^{\alpha}}{\lambda(a,r)}\mu(B(a,r))^{1-\frac{1}{p}} \|\chi_{B(a,r)}\|_{q(\cdot)}\leq 1.
\end{equation}
Then
\begin{equation*}
\frac{(2 K_1)^{\alpha}}{C (C_{\lambda})^{\ell}} \, \frac{r^{\alpha}}{\lambda(a,r)}\mu(B(a,r))^{1-\frac{1}{p}} \|\chi_{B(a,r)}\|_{q_-} \leq\frac{(2 K_1)^{\alpha}}{C (C_{\lambda})^{\ell}} \, \frac{r^{\alpha}}{\lambda(a,r)}\mu(B(a,r))^{1-\frac{1}{p}} \|\chi_{B(a,r)}\|_{q(\cdot)}.
\end{equation*}
But if $\mu(B(a,r))\geq 1$ we have
\begin{equation*}
\frac{(2 K_1)^{\alpha}}{C (C_{\lambda})^{\ell}} \, \frac{r^{\alpha}}{\lambda(a,r)}\mu(B(a,r))^{1-\frac{1}{p}+\frac{1}{q(a)}}\leq \frac{1}{\tilde C} \, \frac{r^{\alpha}}{\lambda(a,r)}\mu(B(a,r))^{1-\frac{1}{p}+\frac{1}{q_-}} \leq 1.
\end{equation*}
If $\mu(B(a,r))\leq 1$ we use in \eqref{estimacionmenoroigual1}, the estimate of $\|\chi_{B(a,r)}\|_{q(\cdot)}$ given in Lemma \ref{acotacionporabajo}, and then we also obtain
$$
\frac{(2 K_1)^{\alpha}}{C (C_{\lambda})^{\ell}} \, \frac{r^{\alpha}}{\lambda(a,r)}\mu(B(a,r))^{1-\frac{1}{p}+\frac{1}{q(a)}}\leq 1
$$
Using that $r^{\alpha}=\lambda(x,r)^{\frac{1}{p}-\frac{1}{q(x)}}$ for all $x\in X$ we get
\begin{equation*}
\mu (B(a,r))^{1-\frac{1}{p}+\frac{1}{q(a)}}\leq \tilde C \frac{\lambda(a,r)}{r^{\alpha}} = \tilde C \lambda(a,r)^{1-\frac{1}{p}+\frac{1}{q(a)}}.
\end{equation*}
From the last inequality we conclude that $\mu(B(x,r))\leq C' \lambda(x,r)$ holds and the proof is complete.
\end{proof}
\begin{remark} When $\lambda(x,r)=r^n$ the condition about the upper type of $\lambda$ implies that $0<\alpha<n$ and the assumption $r^{\alpha}= \lambda(x,r)^{\frac{1}{p}-\frac{1}{q(x)}}$ takes the form $\frac{1}{q}=\frac{1}{p}-\frac{\alpha}{n}$, because from $r^{\alpha}=(r^n)^{\frac{1}{p}-\frac{1}{q(x)}}$ we get that $q$ is constant and satisfies the equality  mentioned above.
\end{remark}

\section{Other example of measure upper doubling}\label{doscomponentes}
One of the authors and Aimar in  \cite{AN} considered the problem of defining a measure when the metric measure space  $X$ is formed by two sets $X_1$ and $X_2$ of different dimensions with certain conditions on its contact. It is easily to state the same result considering quasi-metric spaces instead of metric spaces. If  each component $X_i$  supports  a measure $\mu_i$, $i=1,2$, we can adding this measures and obtain a measure supported on whole space, but $\mu_1+\mu_2$ is not necessarily doubling on $X$. In \cite{AN}   these natural measures $\mu_i$, $i=1,2$  are modified,  introducing  weights distance to contact point, in order to get a doubling measure for the whole space. We will show in this section that the measure defined in \cite{AN} gives an non-trivial example  of upper doubling measure since is not just doubling but variable non-doubling or variable  upper Ahlfors regular. Our setting is  characterized by defining the following three elements:
\begin{enumerate}[{\bf [a]}]
\item \emph{The pieces of $X$.}\\
$X=X_1\cup X_2\cup \{x_0\}$ with $X_1$, $X_2$, $\{x_0\}$ pairwise
disjoint and $(X,d)$ is a bounded metric space.
\item \emph{Contact of order zero.}\\
The components of  $X$ have contact of order zero in $x_0$ or $X$ satisfies the property $\mathcal{C}_{o}$ if and only if
$\{x_0\}=\overline{X}_1\cap \overline{X}_2$ and $d(x,x_0)\leq
c[d(x,X_1) + d(x,X_2)]$ for some
constant $c$ and every $x\in X$.
\item \emph{Dimensions.}\\
$(X_i,d,\mu_i)$ is a Ahlfors $n_i$-regular metric measure space with $0<n_1\leq n_2<\infty$.
\end{enumerate}
Let us observe that since $x_0\in \overline{X_1}\cap\overline{X_2}$ we have the inequality
 $d(x,x_0)\geq d(x,X_1)+
d(x,X_2)$, for every $x\in X$. Hence if $X$ satisfies $\mathcal{C}_{o}$ the constant $c$ is at least one. On the other hand, property $\mathcal{C}_{o}$ provides a pointwise equivalence of the functions $d(x,x_0)$ and $d(x,X_1) + d(x,X_2)$.

It is easy to see that property
$\mathcal{C}_0$ is equivalent to the existence of a constant
$\overline{c}>0$ such that for every $x\in X_i$ is true that $B(x, \overline{c}\ d(x,x_0))\cap X_j= \emptyset,$ $i\neq j$.

As is known, fractal sets produced by the Hutchinson iteration scheme (see \cite{Hu}), under the open set condition, are spaces of homogeneous type with the right Hausdorff measure which are Ahlfors  $Q$-regular for some positive real number $Q$. That is the case of middle thirds Cantor sets and Sierpinsky gaskets.

Our context is a natural abstraction of many situations of fractal fields (see for example \cite{HaKu}) with a special order of contact, we can consider for instance, a plate joined to a block, or a rod joined a plate, or a Cantor set joined a plate, or  a Cantor  set joined a Sierpinsky gasket, etc..

In this context, in  \cite{AN} are introduced ``weights'' to the Ahlfors $n_i$-regular measures, $i=1,2$ in order to get $\mu^{\gamma_1, \gamma_2}$, a doubling measure for the whole space $X=\bigcup_{i=1}^2 X_i\cup\{x_0\}$. More precisely they prove the following theorem.
\begin{theorem}(Theorem 1.2 in \cite{AN}) \label{tortuga}
Assume that $X=X_1\cup X_2\cup \{x_0\}$ satisfies $\mathcal{C}_0$.
For $i=1,2$
 let $(X_{i},d, \mu_i)$ be a $n_i$- normal space with
$0<n_1\leq n_2<\infty$. For $\gamma_{1}
> -n_1$ and $\gamma_{2}
> -n_2$, let  $\mu^{\gamma_{1},\gamma_{2}}$ be the measure defined  by
\begin{equation}
 \mu^{\gamma_{1},\gamma_{2}}(E)= \int_{E \cap X_1}d(x,x_0)^{\gamma_{1}} \, d \mu_1(x) + \int_{E \cap X_2}
d(x,x_0)^{\gamma_{2}}\,
 d \mu_2(x). \label{medidacontact}
\end{equation}
Then $(X,d, \mu^{\gamma_{1},\gamma_{2}})$ is a space of homogenous
type if and only if $\gamma_{1}+n_1=\gamma_{2}+n_2$.\par
\end{theorem}
Throughout this paper we denote by  $\xi$ the number $\gamma_{i}+n_i$, $i=1,2$. It is enough to require
$\xi>0$  to ensure the pairs  $(\gamma_{1},\gamma_{2})$ are admissible for the definition of a doubling measure.
For our purposes, let $\gamma(x)$ denote the function defined in $X$ by
\begin{equation}\label{gammadex}
\gamma(x)= \left\{\begin{array}{ll}
                          \gamma_i, & \hbox{if $x\in X_i$, $i=1,2$;} \\
                          \gamma_1, & \hbox{if $x=x_0$,}
                        \end{array}
                      \right.
\end{equation}
also we could have chosen  $\gamma_2$ as image of $x_0$ and let $n(x)$ denote the function defined in $X$ by
\begin{equation}\label{ndex}
n(x)= \left\{\begin{array}{ll}
                          n_i, & \hbox{if $x\in X_i$, $i=1,2$;} \\
                          n_1, & \hbox{if $x=x_0$,}
                        \end{array}
                      \right.
\end{equation}
or $n_2$ as image of $x_0$.

Although the measure  $\mu^{\gamma_{1},\gamma_{2}}$  is not a Ahlfors $Q$-regular  measure for some $Q>0$, the following theorem give us  estimates of $\mu^{\gamma_{1},\gamma_{2}}$ on balls of $(X,d)$ similar to those of the inequalities of Ahlfors $Q$-regular measure.
\begin{theorem}(Theorem 3.2 in \cite{AN})\label{3-2}
Assume that $(X_1,X_2,d)$  satisfies  $\mathcal{C}_0$. For $i=1,
2$ let $\mu_i$ be  a  Borel measure on $(X_i,d)$ such that
$(X_i,d, \mu_i)$ is a $n_i$-normal space, with $ 0 < n_1 \leq n_2 <
\infty $. For $ \gamma_1 >-n_1$, and $ \gamma_2 > -n_2$, let $
\mu^{\gamma_1,\gamma_2}$ be  the measure define by
(\ref{medidacontact}). Assume that $\gamma_1 - \gamma_2 = n_2 -
n_1$ and  set $\xi = \gamma_1+ n_1 = \gamma_2+ n_2=\gamma(x)+n(x)$. Then, there exist purely geometric constants $ 1\leq K_3 < \infty $  and $1 > c> 0$, such that given $x\in X= X_1\cup X_2 $ and $r>0$ we have
\smallskip
\begin{enumerate}[(i)]
\item
$K_3^{-1}d(x,x_0)^{\gamma(x)}r^{n(x)}\leq
 \mu^{\gamma_1,\gamma_2}(B(x,r))\leq  K_3 d(x,x_0)^{\gamma(x)}r^{n(x)}$,
for $x \in X$ and $r<c\ d(x,x_0)$;

\smallskip
\item
$K_3^{-1}r^{\xi}\leq
 \mu^{\gamma_1,\gamma_2}(B(x,r))\leq
K_3r^{\xi}$, for $c\ d(x,x_0)\leq r \leq S:=\text{diam}\ (X_1)+ \text{diam}\ (X_2)$;

\smallskip
\item
$\mu^{\gamma_1,\gamma_2}(B(x,r))=\mu^{\gamma_,\gamma_2}(X_1)+\mu^{\gamma_1,\gamma_2}(X_2)$,
for $r> S$.
\end{enumerate}
\end{theorem}
We note that it is possible to consider the contact point  $x_0$ as one element of $X$ in the statement of the above theorem; in such case it is easy to obtain, using a density argument, that  $\mu^{\gamma_1,\gamma_2}(B(x_0,r))\sim r^{\xi}$ for all $r>0$.

The estimates in above theorem can be summarized in following way
\begin{equation}\label{ej1updoub}
 K_3^{-1}r^{\xi}\leq\mu^{\gamma_1,\gamma_2}(B(x,r))\leq\left\{
\begin{array}{ll}
     K_3 r^{n(x)}d(x,x_0)^{\gamma(x)} & \hbox{if $r<c\ d(x,x_0)$,}\\
     K_3r^{\xi} & \hbox{if $ r\geq c\ d(x,x_0)$}
\end{array}
\right.
\end{equation}
where $\gamma(x)$ and $n(x)$ denote the functions defined  by \eqref{gammadex} and \eqref{ndex} respectively. Thus \eqref{ej1updoub} provides a new example of measure upper doubling, with upper dominating function defined by
\begin{equation}\label{ej1lambda}
\lambda(x,r)=\left\{
\begin{array}{ll}
   K_3 r^{n(x)}d(x,x_0)^{\gamma(x)} & \hbox{if $r<c\ d(x,x_0)$,}\\
     K_3r^{\xi} & \hbox{if $r\geq c\ d(x,x_0)$}
\end{array}
\right.
\end{equation}
It is easy verify that the function $\lambda$ satisfies the properties required. On the other hand, using  the fact that $(X,d)$  is bounded, we have that $d(x,x_0)$ is less than or equal $R_0={\rm diam\ } X$, we can  obtain another simpler expression for $\lambda$. In fact,
\begin{equation}\label{ej2updoub}
 \mu^{\gamma_1,\gamma_2}(B(x,r))\leq \lambda(x,r)= K_4 r^{n(x)},
 \end{equation}
where $K_4$ is a constant depending of $K_3, R_0, \gamma_1$ and $\gamma_2$. The inequalities \eqref{ej1updoub} and \eqref{ej2updoub} show that $\mu^{\gamma_1,\gamma_2}$ is a measure  lower Ahlfors $\xi$-regular and variable  upper Ahlfors regular, with $\xi\geq n(x)$ for all $x\in X$. As this measure is lower Ahlfors regular there is not exist isolated points and as this measure is doubling is not atomic.

In the setting described by {\bf [a]}, {\bf [b]} and {\bf
[c]}  above, for a function $f$ defined
on $X$ satisfying , $f=f_{1}\chi_{X_1}+ f_2\chi_{X_2}$, belonging to
$L^{1}_{\textrm{loc}}(X,d,\mu^{\gamma_{1},\gamma_{2}})$ we have that its
Hardy- Littlewood  maximal function given by
\begin{align*}\label{maximal}
\mathcal{M}f(x)&=
\sup_{x\in B}\frac{1}{\mu^{\gamma_{1},\gamma_{2}}(B)}\int_{B}|f(y)| \, d
\mu^{\gamma_{1},\gamma_{2}}(y)
\\
&=\sup_{x\in B}\frac{\int_{B\cap X_1}\negthickspace|f_{1}(y)|d(y,x_0)^{\gamma_{1}} \, d
\mu_1(y)
 + \int_{B \cap X_2}\negthickspace|f_{2}(y)|
 d(y,x_0)^{\gamma_{2}}\,
 d\mu_2(y)}{\int_{B\cap
X_1}\negthickspace d(y,x_0)^{\gamma_{1}} \, d \mu_1(y) + \int_{B
\cap X_2} \negthickspace d(y,x_0)^{\gamma_{2}}\, d \mu_2(y)}.
\end{align*}

The maximal operator is bounded on $L^p(X,d,\mu^{\gamma_{1},\gamma_{2}})$ since $\mu^{\gamma_{1},\gamma_{2}}$ is doubling. Moreover $\mathcal{M}$ is bounded on $L^p(wd\mu)$ if and only if $w\in A_p(X,d,\mu^{\gamma_{1},\gamma_{2}})$. In \cite{AIN2} are given necessary and sufficient conditions on two Muckenhoupt $A_p$ weights defined on each of the two components of a space homogeneous type touching at a single point, in order to obtain an $A_p$ weight on the whole space.

\section{Riesz Type Potential in environment doubling with two components of different dimensions}\label{Rieszdoscomponentes}
We can  obtain the Riesz Type Potential operator associated to measure $\mu^{\gamma_{1},\gamma_{2}}$ considering that the upper dominating function in this case is $K_4 r^{n(x)}$. In fact, from \eqref{defIalfamulambda} we get
\begin{equation}
I_{\alpha}^{n(\cdot)} f(x)= \int_{X} \frac{ d(x,y)^{\alpha}}{d(x,y)^{n(x)}} \, f(y) \, d\mu^{\gamma_{1},\gamma_{2}}(y)
\end{equation}
The results about boundedness  of this operator for functions in $L^p(X,d,\mu^{\gamma_{1},\gamma_{2}})$ are a immediate consequence of Theorem \ref{sufIlambda} and Theorem \ref{necIlambda} and lead to the following assertions.
\begin{theorem}\label{sufIndex}
We assume that $X=X_1\cup X_2\cup \{x_0\}$  satisfy $\mathcal{C}_0$ y $(X,d)$ is a quasi-metric bounded. Let  $\mu_i$, $i=1, 2$  be a Borel measure defined on $(X_i,d)$ such that $(X_i,d, \mu_i)$ is a $n_i$-Ahlfors regular space, with $ 0 < n_1 \leq n_2 <\infty $. For $\gamma_1 >-n_1$, and $ \gamma_2 > -n_2$, let $\mu^{\gamma_1,\gamma_2}$ be the measure that satisfies \eqref{ej2updoub} with $\gamma_{1}-\gamma_{2}=n_2-n_1$ and $\gamma_{i}>0$. Let $0<\alpha<n_1$ and
$1<p<q_-\leq q(x)\leq q_+<\infty$ for all $x\in X$, if $\frac{1}{q(x)}=\frac{1}{p}-\frac{\alpha}{n(x)}$ then the operator  $I_{\alpha}^{n(\cdot)}$ is a bounded operator from $L^{p}(X,d, \mu^{\gamma_1,\gamma_2})$ to $L^{q(\cdot)}(X,d, \mu^{\gamma_1,\gamma_2})$.
\end{theorem}
\begin{theorem}\label{necIndex}
For a measure $\mu$, finite over balls and not having any atoms, if $\frac{1}{q(x)}=\frac{1}{p}-\frac{\alpha}{n(x)}$ with $0<\alpha<n_1$ and
$1<p<q_-\leq q(x)\leq q_+<\infty$ for all $x\in X$, then the condition $\mu(B(x,r))\leq \tilde{C} r^{n(x)}$  for some constant $\tilde{C}$ is necessary for $\|I_{\alpha}^{n(\cdot)} f\|_{q(\cdot)}\leq C \|f\|_p$ to hold.
\end{theorem}
\begin{remark}
Note that in the hypotheses of Theorem \ref{sufIndex} is not assumed that $(X,d)$ is geometrically doubling as in Theorem \ref{sufIlambda}, because as mentioned in Section \ref{setting} if $(X,d)$ supports a doubling measure  then $(X,d)$ is geometrically doubling. On the other hand, the condition  about lower type of $\lambda$  when $\lambda(x,r)=r^{n(x)}$ leads to $\alpha<n_1$.
\end{remark}

\bibliographystyle{plain}
\bibliography{ref}

\def\cprime{$'$} \def\cprime{$'$}
\begin{thebibliography}{10}

\bibitem{A1}
H.~Aimar.
\newblock Distance and {M}easure in {A}nalysis and {P}.{D}.{E}.
\newblock to be published by Birkh\"auser.

\bibitem{AIN2}
H.~Aimar, B.~Iaffei, and L.~Nitti.
\newblock Pasting {M}uckenhoupt weights through a contact point between sets of
  different dimensions.
\newblock {\em Acta Math. Hungar.}, 129(4):368--377, 2010.

\bibitem{AN}
Hugo Aimar and Liliana Nitti.
\newblock Separation and contact of sets of different dimensions in a doubling
  environment.
\newblock {\em Publ. Math. Debrecen}, 74(3-4):351--368, 2009.

\bibitem{AS}
Alexandre Almeida and Stefan Samko.
\newblock Fractional and hypersingular operators in variable exponent spaces on
  metric measure spaces.
\newblock {\em Mediterr. J. Math.}, 6(2):215--232, 2009.

\bibitem{AFPf}
{D}. Avnir, {D}. {F}arin, and {P}. {P}feifer.
\newblock Molecular {F}ractal {S}urfaces.
\newblock {\em Nature}, 308:261--263, 1984.

\bibitem{CW}
R.~R. Coifman and G.~Weiss.
\newblock {\em Analyse harmonique non-commutative sur certains espaces
  homog\`enes}.
\newblock Springer-Verlag, Berlin, 1971.

\bibitem{Di2}
Lars Diening.
\newblock Riesz potential and {S}obolev embeddings on generalized {L}ebesgue
  and {S}obolev spaces {$L\sp {p(\cdot)}$} and {$W\sp {k,p(\cdot)}$}.
\newblock {\em Math. Nachr.}, 268:31--43, 2004.

\bibitem{EKM}
David~E. Edmunds, Vakhtang Kokilashvili, and Alexander Meskhi.
\newblock {\em Bounded and compact integral operators}, volume 543 of {\em
  Mathematics and its Applications}.
\newblock Kluwer Academic Publishers, Dordrecht, 2002.

\bibitem{FVA}
D.~Farin, A.~{V}olpert, and D.~{A}vnir.
\newblock Determination of {A}dsorption {C}onformation from {S}urface
  {R}esolution {A}nalysis.
\newblock {\em J. {A}m. {C}hem.}, 107:3368--3370, 1985.

\bibitem{FMS}
Toshihide Futamura, Yoshihiro Mizuta, and Tetsu Shimomura.
\newblock Sobolev embeddings for variable exponent {R}iesz potentials on metric
  spaces.
\newblock {\em Ann. Acad. Sci. Fenn. Math.}, 31(2):495--522, 2006.

\bibitem{GaCuGa}
Jos{\'e} Garc{\'{\i}}a-Cuerva and A.~Eduardo Gatto.
\newblock Boundedness properties of fractional integral operators associated to
  non-doubling measures.
\newblock {\em Studia Math.}, 162(3):245--261, 2004.

\bibitem{GaCuMa01}
Jos{\'e} Garc{\'{\i}}a-Cuerva and Jos{\'e}~Mar{\'{\i}}a Martell.
\newblock Two-weight norm inequalities for maximal operators and fractional
  integrals on non-homogeneous spaces.
\newblock {\em Indiana Univ. Math. J.}, 50(3):1241--1280, 2001.

\bibitem{GaVa1}
A.~Eduardo Gatto and Stephen V{\'a}gi.
\newblock Fractional integrals on spaces of homogeneous type.
\newblock In {\em Analysis and partial differential equations}, volume 122 of
  {\em Lecture Notes in Pure and Appl. Math.}, pages 171--216. Dekker, New
  York, 1990.

\bibitem{GGKK}
Ioseb Genebashvili, Amiran Gogatishvili, Vakhtang Kokilashvili, and Miroslav
  Krbec.
\newblock {\em Weight theory for integral transforms on spaces of homogeneous
  type}, volume~92 of {\em Pitman Monographs and Surveys in Pure and Applied
  Mathematics}.
\newblock Longman, Harlow, 1998.

\bibitem{GPS}
Osvaldo Gorosito, Gladis Pradolini, and Oscar Salinas.
\newblock Boundedness of fractional operators in weighted variable exponent
  spaces with non doubling measures.
\newblock {\em Czechoslovak Math. J.}, 60(135)(4):1007--1023, 2010.

\bibitem{HaKo}
Piotr Haj{\l}asz and Pekka Koskela.
\newblock Sobolev met {P}oincar\'e.
\newblock {\em Mem. Amer. Math. Soc.}, 145(688):x+101, 2000.

\bibitem{HaKu}
B.~M. Hambly and T.~Kumagai.
\newblock Diffusion processes on fractal fields: heat kernel estimates and
  large deviations.
\newblock {\em Probab. Theory Related Fields}, 127(3):305--352, 2003.

\bibitem{HL-I}
G.~H. Hardy and J.~E. Littlewood.
\newblock Some properties of fractional integrals. {I}.
\newblock {\em Math. Z.}, 27(1):565--606, 1928.

\bibitem{HL-II}
G.~H. Hardy and J.~E. Littlewood.
\newblock Some properties of fractional integrals. {II}.
\newblock {\em Math. Z.}, 34(1):403--439, 1932.

\bibitem{HHL}
Petteri Harjulehto, Peter H{\"a}st{\"o}, and Visa Latvala.
\newblock Sobolev embeddings in metric measure spaces with variable dimension.
\newblock {\em Math. Z.}, 254(3):591--609, 2006.

\bibitem{HHP}
Petteri Harjulehto, Peter H{\"a}st{\"o}, and Mikko Pere.
\newblock Variable exponent {L}ebesgue spaces on metric spaces: the
  {H}ardy-{L}ittlewood maximal operator.
\newblock {\em Real Anal. Exchange}, 30(1):87--103, 2004/05.

\bibitem{Hed}
Lars~Inge Hedberg.
\newblock On certain convolution inequalities.
\newblock {\em Proc. Amer. Math. Soc.}, 36:505--510, 1972.

\bibitem{He}
Juha Heinonen.
\newblock {\em Lectures on analysis on metric spaces}.
\newblock Universitext. Springer-Verlag, New York, 2001.

\bibitem{Hu}
John~E. Hutchinson.
\newblock Fractals and self-similarity.
\newblock {\em Indiana Univ. Math. J.}, 30(5):713--747, 1981.

\bibitem{Hy10}
Tuomas Hyt{\"o}nen.
\newblock A framework for non-homogeneous analysis on metric spaces, and the
  {RBMO} space of {T}olsa.
\newblock {\em Publ. Mat.}, 54(2):485--504, 2010.

\bibitem{HM12}
Tuomas Hyt{\"o}nen and Henri Martikainen.
\newblock Non-homogeneous {$Tb$} theorem and random dyadic cubes on metric
  measure spaces.
\newblock {\em J. Geom. Anal.}, 22(4):1071--1107, 2012.

\bibitem{HyYY}
Tuomas Hyt{\"o}nen, Dachun Yang, and Dongyong Yang.
\newblock The {H}ardy space {$H\sp 1$} on non-homogeneous metric spaces.
\newblock {\em Math. Proc. Cambridge Philos. Soc.}, 153(1):9--31, 2012.

\bibitem{KoMe1}
Vakhtang Kokilashvili and Alexander Meskhi.
\newblock Fractional integrals on measure spaces.
\newblock {\em Fract. Calc. Appl. Anal.}, 4(1):1--24, 2001.

\bibitem{KR}
Ondrej Kov{\'a}{\v{c}}ik and Ji{\v{r}}{\'{\i}} R{\'a}kosn{\'{\i}}k.
\newblock On spaces {$L\sp {p(x)}$} and {$W\sp {k,p(x)}$}.
\newblock {\em Czechoslovak Math. J.}, 41(116)(4):592--618, 1991.

\bibitem{LuSa}
Jouni Luukkainen and Eero Saksman.
\newblock Every complete doubling metric space carries a doubling measure.
\newblock {\em Proc. Amer. Math. Soc.}, 126(2):531--534, 1998.

\bibitem{MS79}
R.~A. Mac{\'{\i}}as and C.~Segovia.
\newblock Lipschitz functions on spaces of homogeneous type.
\newblock {\em Adv. in Math.}, 33(3):257--270, 1979.

\bibitem{Mu}
Julian Musielak.
\newblock {\em Orlicz spaces and modular spaces}, volume 1034 of {\em Lecture
  Notes in Mathematics}.
\newblock Springer-Verlag, Berlin, 1983.

\bibitem{Na2}
Eiichi Nakai.
\newblock On generalized fractional integrals in the {O}rlicz spaces on spaces
  of homogeneous type.
\newblock {\em Sci. Math. Jpn.}, 54(3):473--487, 2001.

\bibitem{NTV98}
F.~Nazarov, S.~Treil, and A.~Volberg.
\newblock Weak type estimates and {C}otlar inequalities for
  {C}alder\'on-{Z}ygmund operators on nonhomogeneous spaces.
\newblock {\em Internat. Math. Res. Notices}, (9):463--487, 1998.

\bibitem{Sabook}
Stefan~G. Samko, Anatoly~A. Kilbas, and Oleg~I. Marichev.
\newblock {\em Fractional integrals and derivatives}.
\newblock Gordon and Breach Science Publishers, Yverdon, 1993.
\newblock Theory and applications, Edited and with a foreword by S. M.
  Nikol{\cprime}ski{\u\i}, Translated from the 1987 Russian original, Revised
  by the authors.

\bibitem{Sob}
S.~L. Sobolev.
\newblock {\em On a theorem in functional analysis}, volume~34.
\newblock 1963.

\bibitem{St}
Elias~M. Stein.
\newblock {\em Singular integrals and differentiability properties of
  functions}.
\newblock Princeton Mathematical Series, No. 30. Princeton University Press,
  Princeton, N.J., 1970.

\bibitem{Ta}
H.~Takayasu.
\newblock {\em Fractals in the physical sciences}.
\newblock Nonlinear Science: Theory and Applications. Manchester University
  Press, Manchester, 1990.

\bibitem{TrYa01}
Hans Triebel and Dachun Yang.
\newblock Spectral theory of {R}iesz potentials on quasi-metric spaces.
\newblock {\em Math. Nachr.}, 238:160--184, 2002.

\bibitem{VK2}
A.~L. Vol{\cprime}berg and S.~V. Konyagin.
\newblock On measures with the doubling condition.
\newblock {\em Izv. Akad. Nauk SSSR Ser. Mat.}, 51(3):666--675, 1987.

\bibitem{VW12}
Alexander Volberg and Brett~D. Wick.
\newblock Bergman-type singular integral operators and the characterization of
  {C}arleson measures for {B}esov-{S}obolev spaces and the complex ball.
\newblock {\em Amer. J. Math.}, 134(4):949--992, 2012.

\bibitem{Wu}
J.-M. Wu.
\newblock Hausdorff dimension and doubling measures on metric spaces.
\newblock {\em Proc. Amer. Math. Soc.}, 126(5):1453--1459, 1998.

\bibitem{Za04}
M.~Z{\"a}hle.
\newblock Riesz potentials and {L}iouville operators on fractals.
\newblock {\em Potential Anal.}, 21(2):193--208, 2004.

\bibitem{Za05}
M.~Z{\"a}hle.
\newblock Harmonic calculus on fractals---a measure geometric approach. {II}.
\newblock {\em Trans. Amer. Math. Soc.}, 357(9):3407--3423 (electronic), 2005.

\end{thebibliography}

\end{document}